\documentclass[twoside,leqno,11pt]{article}

\usepackage{amssymb,amsfonts,amsmath}
\usepackage[mathscr]{eucal}
\usepackage[english]{babel}
\usepackage[latin1]{inputenc}
\usepackage{graphicx}
\usepackage{graphics}

\newtheorem{theorem}{Theorem}
\newtheorem{corollary}[theorem]{Corollary}
\newtheorem{definition}[theorem]{Definition}

\newtheorem{lemma}[theorem]{Lemma}

\newtheorem{proposition}[theorem]{Proposition}

\newenvironment{proof}[1][Proof]{\noindent\textbf{#1.} }{\ \rule{0.5em}{0.5em}}

\def\dom{\mathop{\mathrm{dom}}\nolimits}
\def\im{\mathop{\mathrm{im}}\nolimits}
\def\ker{\mathop{\mathrm{ker}}\nolimits}
\def\rank{\mathop{\mathrm{rank}}\nolimits}
\def\id{\mathop{\mathrm{id}}\nolimits}

\def\R{\mathbb R}
\def\T{\mathcal{T}}
\def\S{\mathcal{S}}
\def\E{\mathcal{E}}

\def\O{\mathcal{O}}
\def\OP{\mathcal{OP}}

\title{On relative ranks of the semigroup of orientation-preserving transformations on infinite chain with restricted range}

\author{Ilinka Dimitrova\footnote{The author gratefully acknowledges support of DAAD, within the funding programme "Research Stays for University Academics and Scientists, 2019" number 57442043.} \\
\textit{Faculty of Mathematics and Natural Science}\\
\textit{South-West University "Neofit Rilski"}\\
\textit{2700 Blagoevgrad, Bulgaria}\\
\textit{e-mail: ilinka\_dimitrova@swu.bg}\\
~~\\
J\"{o}rg Koppitz \\
\textit{Institute of Mathematics and Informatics}\\
\textit{Bulgarian Academy of Sciences}\\
\textit{1113 Sofia, Bulgaria}\\
\textit{e-mail: koppitz@math.bas.bg}\\
~~\\
\textit{Institute of Mathematics}\\
\textit{Potsdam University}\\
\textit{Potsdam, 14469, Germany}\\
\textit{koppitz@uni-potsdam.de}\\}

\begin{document}

\maketitle

\begin{abstract}
Let $X$ be an infinite linearly ordered set and let $Y$ be a nonempty subset of $X$.
We calculate the relative rank of the semigroup $\OP(X,Y)$ of all orientation-preserving transformations on $X$ with restricted range $Y$
modulo the semigroup $\O(X,Y)$ of all order-preserving transformations on $X$ with restricted range $Y$. For $Y = X$, we characterize the relative generating
sets of minimal size.
\end{abstract}


\\

\section{Introduction and Preliminaries}

Let $S$ be a semigroup. The \textit{rank} of $S$ (denoted by $\rank S$) is defined to be the minimal number of elements of a generating set of $S$. The ranks of various well known semigroups have been calculated \cite{GH, GH1, HMcF}.
For a set $A \subseteq S$, the \textit{relative rank} of $S$ modulo $A$, denoted by $\rank (S : A)$, is the minimal cardinality of a set $B\subseteq S$ such that $A \cup B$ generates $S$. It follows immediately from the definition that $\rank (S : \emptyset) = \rank S$, $\rank ( S : S) = 0$, $\rank (S : A ) = \rank (S : \langle A \rangle )$ and $\rank (S : A) = 0$ if and only if $A$ is a generating set for $S$.
The relative rank of a semigroup modulo a suitable set was first introduced by Ru\v skuc \cite{Ruskuc} in order to describe the generating sets of semigroups with infinite rank.

A total (i.e., a linear) order $<$ on a set $X$ is said to be \textit{dense} if, for all $x, y \in X$ with $x < y$, there is a $z \in X$ such that $x < z < y$.
Let $X$ be a densely infinite linearly ordered set.
We denote by $\T(X)$ the monoid of all full transformations of $X$ (under composition).
In \cite{HRH}, Howie, Ru\v skuc, and Higgins  considered the relative ranks of the monoid $\T(X)$, where $X$ is an infinite set, modulo some distinguished subsets of $\T(X)$.
They showed that $\rank (\T(X) : \S(X)) = 2$, $\rank (\T(X) : \E(X)) = 2$ and $\rank (\T(X) : J) = 0$, where $\S(X)$ is the symmetric group on $X$, $\E(X)$ is the set of all idempotent transformations on $X$, and $J$ is the top $\mathcal{J}$-class of $\T(X)$, i.e., $J=\{\alpha\in\T(X) : |X\alpha|=|X|\}$.

A function $f : A \rightarrow X$ from a subchain $A$ of $X$ into $X$ is said to be \textit{order-preserving} if $x \leq y$ implies $xf \leq yf$, for all $x, y \in A$. Notice that, given two subchains $A$ and $B$ of $X$ and an order-isomorphism (i.e., an order-preserving bijection) $f : A \rightarrow B$, the inverse function $f^{-1} : B \rightarrow A$ is also an order-isomorphism. In this case, the subchains $A$ and $B$ are called \textit{order-isomorphic}. We denote by $\O(X)$ the submonoid of $\T(X)$ of all order-preserving transformations of $X$.
The relative rank of $\T(X)$ modulo $\O(X)$ was considered by Higgins, Mitchell, and Ru\v{s}kuc \cite{HMR}.
They showed that $\rank(\T(X) : \O(X)) = 1$, when $X$ is an arbitrary countable chain or an arbitrary well-ordered set,
while $\rank(\T(\R) : \O(\R))$ is uncountable, by considering the usual order of the set $\R$ of real numbers.
In \cite{DFK}, Dimitrova, Fernandes, and Koppitz studied the relative rank of the semigroup $\O(X)$ modulo the top $\mathcal{J}$-class of $\O(X)$,
for an infinite countable chain $X$.

A generalization of the concept of an order-preserving transformation is the
concept of an orientation-preserving transformation, which was introduced in
1998 by McAlister \cite{McA} and, independently, one year later by Catarino and
Higgins \cite{CH}, but only for finite chains. In \cite{FJS}, Fernandes, Jesus, and Singha
introduced the concept of an orientation-preserving transformation on an
infinite chain. It generalizes the concept for a finite chain.
\begin{definition} \rm \cite{FJS}
Let $\alpha \in \T(X)$. We say that $\alpha$ is an \textit{orientation-preserving} transformation if there exists a non-empty subset $X_1$ of $X$ such that:\\
(1) $\alpha$ is order-preserving both on $X_1$ and on $X_2 = X\setminus X_1$; and\\
(2) for all $a \in X_1$ and $b \in X_2$, we have $a < b$ and $a\alpha \geq b\alpha$.
\end{definition}
In the present paper, we will keep this meaning of $X_1$ and $X_2$ for an orientation-preserving transformation $\alpha$ if it is clear from the
context. Since $x \leq a$ implies $x \in X_1$ for all $x \in X$ and all $a \in X_1$, the non-empty set $X_1$ is an (order) ideal of $X$.
We call the set $X_1$ an \textit{ideal of $\alpha$}.

Denote by $\OP(X)$ the subset of $\T(X)$ of all orientation-preserving transformations. In \cite{FJS}, Fernandes, Jesus, and Singha proved that $\OP(X)$ is a monoid. Moreover, they proved that if $\alpha \in \OP(X)$ is a non-constant transformation then $\alpha$ admits a unique ideal.
Clearly, $\O(X) \subseteq \OP(X)$ and we have $\alpha \in \O(X)$ if and only if $\alpha \in \OP(X)$ and $\alpha$ admits $X$ as an ideal.
In \cite{DK2}, Dimitrova and Koppitz determined the relative rank of the semigroup $\OP(X)$ modulo $\O(X)$ for some infinite chains $X$.

Given a nonempty subset $Y$ of $X$, we denote by $\T(X,Y)$ the subsemigroup $\{\alpha \in \T(X) : \im(\alpha) \subseteq Y\}$ of $\T(X)$ of all elements with range (image) in $Y$.
The semigroup $\T(X,Y)$ was introduced and studied in 1975 by Symons (see \cite{Symons}), and it is called semigroup of transformations with restricted range.
Recently, for a finite set $X$ the rank of $\T(X,Y)$ was computed by Fernandes and Sanwong \cite{FS}.

In this paper, for the set $X$ and a nonempty subset $Y$ of $X$, we consider the order-preserving and the orientation-preserving counterparts of the semigroup $\T(X,Y)$, namely the semigroup $\O(X,Y) = \T(X,Y) \cap \O(X) = \{\alpha \in \O(X) : \im(\alpha) \subseteq Y\}$ and the semigroup $\OP(X,Y) = \T(X,Y) \cap \OP(X) = \{\alpha \in \OP(X) : \im(\alpha) \subseteq Y\}$.

For a finite chain $X$ the ranks of the semigroups $\O(X,Y)$ and $\OP(X,Y)$ were determined by Fernandes, Honyam, Quinteiro, and Singha \cite{FHQS1, FHQS2}.
In \cite{TK1}, Tinpun and Koppitz studied the relative rank of $\T(X,Y)$ modulo $\O(X,Y)$.
In \cite{DKT}, Dimitrova, Koppitz, and Tinpun studied rank properties of the semigroup $\OP(X,Y)$. In \cite{DK1}, Dimitrova and Koppitz determined the relative rank of $\T(X,Y)$ modulo $\OP(X,Y)$ and the relative rank of the semigroup $\OP(X,Y)$ modulo $\O(X,Y)$.
In \cite{TK2}, Tinpun and Koppitz considered generating sets of infinite full transformation semigroups with restricted range.

We begin by recalling some notations and definitions that will be used in the paper.
For any transformation $\alpha \in \T(X,Y)$, we denote by $\ker \alpha$, $\dom \alpha$, and $\im \alpha$ the kernel, the domain, and the image (range) of $\alpha$, respectively.
The inverse of $\alpha$ as a relation is denoted by $\alpha^{-1}$. For a subset $A \subseteq \T(X,Y)$, we denote by $\langle A \rangle$ the subsemigroup of $\T(X,Y)$ generated by $A$. For a subset $C \subseteq X$, we denote by $\alpha|_C$ the restriction of $\alpha$ to $C$ and by $\id_C$ the identity mapping on $C$.
A subset $C$ of $X$ is called a \textit{convex} subset of $X$ if $z \in X$ and $x < z < y$ imply $z \in C$, for all $x,y \in C$.
Let $C$ and $D$ be convex subsets of $X$. We will write $C < D$ (respectively, $C \leq D$) if $c < d$ (respectively, $c \leq d$) for all $c \in C$ and all $d \in D$. If $C = \{c\}$ or $D = \{d\}$, we write $c < D$, $C < d$ (respectively, $c \leq D$, $C \leq d$) instead of $\{c\} < D$ or $C < \{d\}$ (respectively, $\{c\} \leq D$ or $C \leq \{d\}$). Note that $C < \emptyset$ and $\emptyset < D$.

For $A < B \subseteq X$ and $a, b \in X$ with $a < b$, $a < B$, and $A < b$, we put
$$(A,B) = \{x \in X : A < x < B\},$$
for a convex subset of $X$ which has no minimum and no maximum;
$$[a,B) = \{x \in X : a \leq x < B\},$$
for a convex subset of $X$ which has a minimum $a$ but no maximum;
$$(A,b] = \{x \in X : A < x \leq b\},$$
for a convex subset of $X$ which has no minimum but a maximum $b$; and
$$[a,b] = \{x \in X : a \leq x \leq b\},$$
for a convex subset of $X$ which has a minimum $a$ and a maximum $b$.

A convex subset of $X$ without minimum and maximum will be called \textit{open convex subset} of $X$.

Notice that if $A = \emptyset$ or $B = \emptyset$ then we have
$$(A,\emptyset) = \{x \in X : A < x\}, ~~~ (\emptyset,B) = \{x \in X : x < B\},$$
$$[a,\emptyset) = \{x \in X : a \leq x \}, ~\mbox{ and }~ (\emptyset,b] = \{x \in X : x \leq b\}.$$

\begin{proposition}\label{pr1} \rm \cite{FJS}
Let $\alpha \in \OP(X)$ with ideal $X_1$. If $X_1\alpha \cap X_2\alpha \neq \emptyset$ then $X_1\alpha \cap X_2\alpha = \{c\}$, for some $c \in X$. Moreover, in this case, $X_1\alpha$ has a minimum, $X_2\alpha$ has a maximum, and both of these elements coincide with $c$.
\end{proposition}

We consider a densely infinite linearly ordered set $X$ with the following two properties:
\begin{description}
  \item[(a)] if $X = X' \cup X''$ with $X' < X''$ and $X', X'' \neq \emptyset$ is a decomposition of $X$ then $X'$ has a maximum or $X''$ has a minimum;
  \item[(b)] any two open convex subsets of $X$ are order-isomorphic.
\end{description}

The set $\mathbb{R}$ of real numbers under the usual order $\leq$ is a nontrivial example for a densely infinite linearly ordered set satisfying the conditions (a) and (b).
In fact, any two open convex subsets of $\mathbb{R}$ (open intervals) are order-isomorphic. Moreover, if $\mathbb{R} = X' \cup X''$ with $X' < X''$ is a decomposition of $\mathbb{R}$ then there is an $a \in \mathbb{R}$ such that either $X' = \{x \in \mathbb{R} \mid x \leq a\}$ and $X'' = \{x \in \mathbb{R} \mid x > a\}$ or
$X' = \{x \in \mathbb{R} \mid x < a\}$ and $X'' = \{x \in \mathbb{R} \mid x \geq a\}$. In that case, $a$ is the maximum of $X'$ and the minimum of $X''$, respectively.\\

In the present paper, we determine the relative rank of $\OP(X,Y)$ modulo $\O(X,Y)$ for a nonempty proper subset $Y$ of $X$.
We consider two cases for the set $Y$: (1) $Y$ is a convex subset of $X$; (2) $Y$ is not a convex subset of $X$ but $Y$ contains an open convex subset of $X$.
For $Y = X$, we characterize the relative generating sets of minimal size.

\begin{lemma}\label{le1} \rm
Let $\alpha \in \OP(X,Y)\setminus \O(X,Y)$ with ideal $X_1$ and let $c \in X$ be the maximum of $X_1$ (respectively the minimum of $X_2$) then $c\alpha \in Y$ is the maximum of $\im\alpha$ (respectively $c\alpha \in Y$ is the minimum of $\im\alpha$).
\end{lemma}
\begin{proof}
Let $c \in X$ be the maximum of $X_1$ and let $y \in \im\alpha \subseteq Y$. Then there is $x \in X$ such that $x\alpha = y$. If $x \in X_1$ then $x\alpha \leq c\alpha$, since $c$ is the maximum of $X_1$ and $\alpha |_{X_1}$ is order-preserving. If $x \in X_2$ then $x\alpha \leq z\alpha$ for all $z \in X_1$, since $X_2\alpha \leq X_1\alpha$ ($\alpha$ is orientation-preserving). Therefore, we obtain $y\leq c\alpha$, i.e., $c\alpha$ is the maximum of $\im\alpha$.

Now, let $c \in X$ be the minimum of $X_2$. Dually, one can obtain that $c\alpha$ is the minimum of $\im\alpha$.
\end{proof}
~~\\

From Lemma \ref{le1}, we obtain:
\begin{corollary}\label{co1} \rm
Suppose $Y \subset X$ has no minimum and no maximum. Then for all $\alpha \in \OP(X,Y)\setminus \O(X,Y)$
there exists $h \in Y$ with $\{y \in Y : y > h\}\cap \im\alpha =\emptyset$ or there exists $l \in Y$ with $\{y \in Y : y < l\} \cap \im\alpha =\emptyset$.
\end{corollary}
\begin{proof}
Let $c \in X$ be the maximum of $X_1$. From Lemma  \ref{le1}, we have $c\alpha \in Y$ is the maximum of $\im\alpha$. Since $\im\alpha \subseteq Y$ and $Y$ has no maximum, it follows that $\{y \in Y : y > c\alpha\}\cap \im\alpha =\emptyset$.

Now, let $c \in X$ be the minimum of $X_2$. Dually, one can obtain that $\{y \in Y : y < c\alpha\}\cap \im\alpha =\emptyset$.
\end{proof}

\section{The relative generating sets of $\OP(X)$ modulo $\O(X)$ of minimal size}

In \cite{DK2}, Dimitrova and Koppitz determined the relative rank of the semigroup $\OP(X)$ modulo $\O(X)$ for certain infinite chains $X$. It remains a
characterization of the relative generating sets of $\OP(X)$ modulo $\O(X)$ of minimal size. This will be the main purpose of this section. In particular, this result will be used in the next section. Recall, Dimitrova and Koppitz have already proved:
\begin{proposition}\label{pr2} \rm \cite{DK2}\\
1) Let the set $X$ have a minimum $a$ and a maximum $b$, and let $c \in (a,b)$ and $d \in (c,b)$, i.e., $a < c < d < b$.
Then $\OP(X) = \langle \O(X),\gamma \rangle$, where
$$x\gamma = \left\{
             \begin{array}{ll}
               d, & x = a \\
               x\mu_1, & x \in (a,c) \\
               a, & x = c \\
               x\mu_2, & x \in (c,b) \\
               c, & x = b \\
             \end{array}
           \right.$$
with the order-isomorphisms $\mu_1 : (a,c) \rightarrow (d,b)$ and $\mu_2 : (c,b) \rightarrow (a,c)$.\\
\\
2) Let the set $X$ have a minimum $a$ but no maximum and let $c \in (a,\emptyset)$.
Then $\OP(X) = \langle \O(X),\gamma \rangle$, where
$$x\gamma = \left\{
             \begin{array}{ll}
               c, & x = a \\
               x\nu, & x \in (a,c) \\
               a, & x = c \\
               x\nu^{-1}, & x \in (c,\emptyset) \\
             \end{array}
           \right.$$
with the order-isomorphism $\nu : (a,c) \rightarrow (c,\emptyset)$.\\
\\
3) Let the set $X$ have no minimum but a maximum $b$ and let $c \in (\emptyset,b)$ and $l \in (\emptyset,c)$, i.e., $l < c < b$.
Then $\OP(X) = \langle \O(X),\gamma \rangle$, where
$$x\gamma = \left\{
             \begin{array}{ll}
               x\tau_1, & x \in (\emptyset,c) \\
               l, & x = c \\
               x\tau_2, & x \in (c,b) \\
               c, & x = b \\
             \end{array}
           \right.$$
with the order-isomorphisms $\tau_1 : (\emptyset,c) \rightarrow (c,b)$ and $\tau_2 : (c,b) \rightarrow (l,c)$.\\
\end{proposition}

In all the cases, $\rank(\OP(X):\O(X)) = 1$ and $\{\gamma\}$ is a
relative generating set of $\OP(X)$ modulo $\O(X)$. So, we
have to indicate all singleton sets which are relative generating sets of $\OP(X)$ modulo $\O(X)$.
First, we consider the case where $X$ has a minimum and a maximum.
\begin{proposition}\label{pr3}  \rm
Suppose $X$ has a minimum $a$ and a maximum $b$, i.e., $X = [a,b]$.
Let $\varphi \in \OP(X)\setminus \O(X)$ with the ideal $X_{1}$. Then $\left\langle
\O(X),\varphi \right\rangle = \OP(X)$ if and only if, for $i \in \{1,2\}$, there
is a set $Y_{i}\subseteq X_{i}$ which is order-isomorphic to an open convex
subset of $X$ such that $\varphi|_{Y_{i}}$ is an order-isomorphism.
\end{proposition}

\begin{proof}
Suppose that for $i \in \{1,2\}$, there is a set $Y_{i}\subseteq X_{i}$ which
is order-isomorphic to an open convex subset of $X$ such that $\varphi|_{Y_{i}}$ is an order-isomorphism.
Let $c, d \in X$ such that $a < c < d < b$. Then there are sets $U_{1}\subseteq Y_{1}$, order-isomorphic to $[a,c)$, and $U_{2}\subseteq Y_{2}$, order-isomorphic to $[c,b]$, with respect to the order-isomorphisms $\mu_{1}$ and $\mu_{2}$, respectively. Then we define $\widehat{\varphi} : X \rightarrow X$ by
$$x\widehat{\varphi} = \left\{
\begin{array}{ll}
x\mu_{1}, &  x \in [a,c) \\
x\mu_{2}, &  x \in [c,b].
\end{array}
\right.$$
Clearly, $\widehat{\varphi} \in \O(X)$. Further, there are order-isomorphisms
\begin{eqnarray*}
\mu_{3}: (a,c)\mu_{1}\varphi &\rightarrow &(d,b) \\
\mu_{4}: (c,b)\mu_{2}\varphi &\rightarrow &(a,c).
\end{eqnarray*}%
We define $\widetilde{\varphi}: X \rightarrow X$ by
$$x\widetilde{\varphi} = \left\{
\begin{array}{ll}
a, & x \in [a,c\mu_{2}\varphi] \\
x\mu_{4}, & x \in (c,b)\mu_{2}\varphi \\
c, & x \in [b\mu_{2}\varphi, a\mu_{1}\varphi) \\
d, & x = a\mu_{1}\varphi \\
x\mu_{3}, & x \in (a,c)\mu_{1}\varphi \\
b, & x \in [c\mu_{1}\varphi,b].
\end{array}
\right.$$
Note that $(c,b)\mu_{2}\varphi < (a,c)\mu_{1}\varphi$, since $\varphi \in \OP(X)\setminus \O(X)$, $(a,c)\mu_{1} \subseteq X_{1}$, and $(c,b)\mu_{2} \subseteq X_{2}$.
Since $\mu_1, \mu_2, \mu_3, \mu_4$ as well as $\varphi|_{Y_{i}}$ ($i=1,2$) are order-isomorphisms and
$$\min X = a \leq c\mu_{2}\varphi < (c,b)\mu_{2}\varphi < [b\mu_{2}\varphi, a\mu_{1}\varphi) <  a\mu_{1}\varphi < (a,c)\mu_{1}\varphi < c\mu_{1}\varphi \leq b = \max X,$$
$$a < (a,c) = (c,b)\mu_{2}\varphi\mu_{4} < c < d < (d,b) = (a,c)\mu_{1}\varphi\mu_{3} < b,$$
we have that $\widetilde{\varphi} \in \O(X)$.

Further, we have
$$(a,c)\widehat{\varphi}\varphi\widetilde{\varphi} = (a,c)\mu_{1}\varphi\widetilde{\varphi} = (a,c)\mu_{1}\varphi\mu_{3}$$
and
$$(c,b)\widehat{\varphi}\varphi\widetilde{\varphi} = (c,b)\mu_{2}\varphi\widetilde{\varphi} = (c,b)\mu_{2}\varphi\mu_{4}.$$
This shows that
$$\widehat{\varphi}\varphi \widetilde{\varphi}|_{(a,c)} : (a,c) \rightarrow (d,b)$$
and
$$\widehat{\varphi}\varphi\widetilde{\varphi}|_{(c,b)} : (c,b) \rightarrow (a,c)$$
are order-isomorphisms. Moreover, we have
$a\widehat{\varphi}\varphi\widetilde{\varphi} = (a\widehat{\varphi})\varphi\widetilde{\varphi} = (a\mu_1)\varphi\widetilde{\varphi} = (a\mu_1\varphi)\widetilde{\varphi} = d$,
$c\widehat{\varphi}\varphi\widetilde{\varphi} = = (c\widehat{\varphi})\varphi\widetilde{\varphi} = (c\mu_2)\varphi\widetilde{\varphi} = (c\mu_2\varphi)\widetilde{\varphi} a$, and $b\widehat{\varphi}\varphi\widetilde{\varphi} = (b\widehat{\varphi})\varphi\widetilde{\varphi} = (b\mu_2)\varphi\widetilde{\varphi} = (b\mu_2\varphi)\widetilde{\varphi} = c$.
Then $\widehat{\varphi}\varphi\widetilde{\varphi}$ is the transformation $\gamma$ from Proposition \ref{pr2} case 1), which is used as relative generating
set for $\OP(X)$ modulo $\O(X)$. Since $\widehat{\varphi}\varphi\widetilde{\varphi} \in \left\langle \O(X), \varphi \right\rangle$, we can conclude that
$\left\langle \O(X), \varphi \right\rangle = \OP(X)$.\\

Conversely, suppose now that $\left\langle \O(X), \varphi \right\rangle = \OP(X)$. Suppose by way of contradiction that
there is $i \in \{1,2\}$ such that there is no subset $Z$ of $X_{i}$
which is order-isomorphic to an open convex subset of $X$ such that $\varphi|_{Z}$ is an order-isomorphism.
Let $p,q \in X$ with $a < p < q < b$. Further, let $\sigma_{1}: [a,p) \rightarrow [q,b)$ and let
$\sigma_{2}: [p,b] \rightarrow [a,p]$ be order-isomorphisms. Then let $\alpha: X \rightarrow X$ be defined by
$$x\alpha = \left\{
\begin{array}{ll}
x\sigma_{1} & x \in [a,p) \\
x\sigma_{2} & x \in [p,b].
\end{array}%
\right.$$
Clearly, $\alpha \in \OP(X)\setminus \O(X)$ and there are $\alpha_{1}, \ldots, \alpha_{n} \in \O(X) \cup \{\varphi\}$ such that
$\alpha = \alpha_{1} \cdots \alpha_{n}$. Moreover, we put $\alpha_{0} = \id_X \in \O(X)$.

Assume that $|X_{i}\cap \im(\alpha_{0} \cdots \alpha_{j-1})| < \aleph_{0}$ for all $j \in \{1,\ldots,n\}$ with
$\alpha_{j} = \varphi$. Then we obtain that the ideal of $\alpha = \alpha_{1} \cdots \alpha_{n}$ is finite or only finite elements do not belong to the ideal.
This is a contradiction, since the ideal of $\alpha$ is the infinite set $[a,p)$ and the infinite set $[p,b]$ does not belong to the ideal of $\alpha$.
Hence there is $k \in \{1,\ldots,n\}$ such that $|X_{i}\cap \im(\alpha_{0} \cdots \alpha_{k-1})| = \aleph_{0}$ and $\alpha_{k}=\varphi$.
Then there are $x_{1} < x_{2} \in X_{i} \cap
\im(\alpha_{0} \cdots \alpha_{k-1})$ such that $\alpha_{0} \cdots \alpha_{k-1}|_{\widetilde{Z}}$ is order-preserving, where
$$\widetilde{Z} = \{x(\alpha_{0}\cdots \alpha_{k-1})^{-1} : x \in (x_{1},x_{2}) \cap \im(\alpha_{0} \cdots \alpha_{k-1})\}.$$
Since $X_{i}$ is convex, we have $\widehat{Z} = (x_{1},x_{2})\cap \im(\alpha_{0} \cdots \alpha_{k-1}) \subseteq X_{i}$.
Moreover, since $\alpha = \alpha_{1} \cdots \alpha_{n}$ is injective, we obtain that $\alpha_{k}|_{\widehat{Z}}$ is an order-isomorphism. On the other hand, $\widetilde{Z}$ is a subset of $X_i$ and thus $\widetilde{Z}$ is not order-isomorphic to any open convex subset of $X$.
Hence, $\widetilde{Z}$ is not a convex set, since $\widehat{Z}$ is order-isomorphic to $\widetilde{Z}$. Thus, there are $x_{3} < x_{4} \in \widetilde{Z}$ and $z \in (x_{3},x_{4})\setminus \widetilde{Z}$.
We have $x_{3} < z < x_{4}$ and $z$ belongs to the ideal of $\alpha_{0} \cdots \alpha_{k-1}$ if and only if $\widetilde{Z}$ belongs to it.
This implies $x_{3}(\alpha_{0} \cdots \alpha_{k-1}) \leq z(\alpha_{0} \cdots \alpha_{k-1}) \leq x_{4}(\alpha_{0} \cdots \alpha_{k-1})$, since
$\alpha_{0} \cdots \alpha_{k-1}$ restricted to its ideal (and its complement, respectively) is order-preserving.
Clearly, $z(\alpha_{0} \cdots \alpha_{k-1})\in \im(\alpha_{0} \cdots \alpha_{k-1})$.
Moreover, there are $\widetilde{x}_{3}, \widetilde{x}_{4} \in \widehat{Z}$ such that
$x_{3} = \widetilde{x}_{3}(\alpha_{0} \cdots \alpha_{k-1})^{-1}$ and $x_{4} = \widetilde{x}_{4}(\alpha_{0} \cdots \alpha_{k-1})^{-1}$.
This shows that $x_{1} < \widetilde{x}_{3} = x_{3}(\alpha_{0} \cdots \alpha_{k-1}) < z(\alpha_{0} \cdots \alpha_{k-1}) < x_{4}(\alpha_{0} \cdots \alpha_{k-1}) = \widetilde{x}_{4} < x_{2}$, i.e., $z(\alpha_{0} \cdots \alpha_{k-1}) \in \widehat{Z}$. Thus $z \in \widetilde{Z}$, a contradiction.
\end{proof}
~~\\

Now, we consider the case where $X$ has a minimum but no maximum.
\begin{proposition}\label{pr4}  \rm
Suppose $X$ has a minimum $a$ but no maximum, i.e., $X = [a,\emptyset)$. Let $\varphi \in \OP(X)\setminus \O(X)$ with the ideal $X_{1}$.
Then $\left\langle \O(X), \varphi \right\rangle = \OP(X)$ if and only if, for $i\in \{1,2\}$, there
is a set $Y_{i}\subseteq X_{i}$ which is order-isomorphic to an open convex
subset of $X$ such that $\varphi|_{Y_{i}}$ is an order-isomorphism and $(Y_{1}\varphi,\emptyset) = \emptyset$.
\end{proposition}

\begin{proof}
Suppose that for $i \in \{1,2\}$, there is a set $Y_{i}\subseteq X_{i}$ which is order-isomorphic to an open convex subset of $X$ such that $\varphi|_{Y_{i}}$ is an order-isomorphism and $(Y_{1}\varphi,\emptyset)=\emptyset$. Let $c \in (a,\emptyset)$. Then there are sets $U_{1}\subseteq Y_{1}$ with $(U_{1}\varphi,\emptyset) = \emptyset$, order-isomorphic to $[a,c)$, and $U_{2}\subseteq Y_{2}$, order-isomorphic to $[c,\emptyset)$, with respect to the order-isomorphisms $\mu_{1}$ and $\mu_{2}$, respectively.
Then we define $\widehat{\varphi}: X \rightarrow X$ by
$$x\widehat{\varphi} = \left\{
\begin{array}{ll}
x\mu_{1}, & x \in [a,c) \\
x\mu_{2}, & x \in [c,\emptyset).
\end{array}
\right.$$
Clearly, $\widehat{\varphi} \in \O(X)$. Further, there are order-isomorphisms
\begin{eqnarray*}
\mu_{3}: (a,c)\mu_{1}\varphi & \rightarrow & (c,\emptyset) \\
\mu_{4}: (c,\emptyset)\mu_{2}\varphi & \rightarrow & (a,c).
\end{eqnarray*}%
We define $\widetilde{\varphi}: X \rightarrow X$  by
$$x\widetilde{\varphi} = \left\{
\begin{array}{ll}
a, & x \in [a,c\mu_{2}\varphi] \\
x\mu_{4}, & x \in (c,\emptyset)\mu_{2}\varphi \\
c, & x\in ((c,\emptyset)\mu_{2}\varphi, a\mu_{1}\varphi] \\
x\mu_{3}, & x \in (a,c)\mu_{1}\varphi.
\end{array}
\right.$$
The transformation $\widetilde{\varphi}$ is well defined since $(U_{1}\varphi,\emptyset) = \emptyset$ and $(c,\emptyset)\mu_{2}\varphi < (a,c)\mu_{1}\varphi$ (since $\varphi \in \OP(X)\setminus \O(X)$).
Since $\mu_1, \mu_2, \mu_3, \mu_4$ as well as $\varphi|_{Y_{i}}$ ($i=1,2$) are order-isomorphisms and
$$\min X = a \leq c\mu_{2}\varphi < (c,\emptyset)\mu_{2}\varphi < ((c,\emptyset)\mu_{2}\varphi, a\mu_{1}\varphi] <  (a,c)\mu_{1}\varphi = U_{1}\varphi,$$
$$a < (a,c) = (c,\emptyset)\mu_{2}\varphi\mu_{4} < c < (c,\emptyset) = (a,c)\mu_{1}\varphi\mu_{3},$$
we have that $\widetilde{\varphi} \in \O(X)$.

Further, we have
$$(a,c)\widehat{\varphi}\varphi\widetilde{\varphi} = (a,c)\mu_{1}\varphi\widetilde{\varphi} = (a,c)\mu_{1}\varphi\mu_{3}$$
and
$$(c,\emptyset)\widehat{\varphi}\varphi\widetilde{\varphi} = (c,\emptyset)\mu_{2}\varphi\widetilde{\varphi} = (c,\emptyset)\mu_{2}\varphi\mu_{4}.$$
This shows that
$$\widehat{\varphi}\varphi\widetilde{\varphi}|_{(a,c)} : (a,c) \rightarrow (c,\emptyset)$$
and
$$\widehat{\varphi}\varphi\widetilde{\varphi}|_{(c,\emptyset)} : (c,\emptyset) \rightarrow (a,c)$$
are order-isomorphisms. Moreover, we have $a\widehat{\varphi}\varphi\widetilde{\varphi} = (a\widehat{\varphi})\varphi\widetilde{\varphi} = (a\mu_1)\varphi\widetilde{\varphi} = (a\mu_1\varphi)\widetilde{\varphi} = c$ and $c\widehat{\varphi}\varphi\widetilde{\varphi} = (c\widehat{\varphi})\varphi\widetilde{\varphi} = (c\mu_2)\varphi\widetilde{\varphi} = (c\mu_2\varphi)\widetilde{\varphi} = a$.
Then $\widehat{\varphi}\varphi\widetilde{\varphi}$ is the transformation $\gamma$ from Proposition \ref{pr2} case 2), which is used as relative generating
set for $\OP(X)$ modulo $\O(X)$.
Since $\widehat{\varphi}\varphi\widetilde{\varphi} \in \left\langle \O(X), \varphi \right\rangle$, we can conclude that $\left\langle \O(X), \varphi \right\rangle = \OP(X)$.\\

Suppose now that $\left\langle \O(X), \varphi \right\rangle = \OP(X)$.
Let $\alpha \in \OP(X)\setminus \O(X)$ with the ideal $X_{1}$ such that $\alpha|_{X_{1}}$ is injective, $X_{1}\alpha$ is convex, and $(X_{1}\alpha,\emptyset)=\emptyset$.
Then there are $\alpha_{1},\ldots ,\alpha_{n}\in \O(X)\cup \{\varphi\}$ such that $\alpha = \alpha_{1}\cdots \alpha_{n}$. Since $\alpha \notin \O(X)$,
there is $k \in \{1,\ldots,n\}$ with $\alpha_{k} = \varphi$. Without loss of generality, we can assume that $k$ is the greatest $r \in \{1,\ldots,n\}$
with $\alpha_{r}=\varphi$ and $\im(\alpha_{0}\cdots \alpha_{r-1}) \cap X_1 \neq \emptyset$, where $\alpha_0 = \id_X$.

We put $Z_{1} = \im(\alpha_{0}\cdots \alpha_{k-1})\cap X_{1}$ and observe that $Z_{1}(\alpha_{k}\cdots \alpha_{n})\geq (\im(\alpha_{0}\cdots \alpha_{k-1})\setminus Z_{1})(\alpha_{k}\cdots \alpha_{n})$. Let $m \in Z_{1}(\alpha_{k}\cdots \alpha_{n})$. Since $(X_{1}\alpha,\emptyset)=\emptyset$, we can conclude that $(m,\emptyset)\subseteq \im\alpha$, since $X_{1}\alpha$ is convex and $(X_{1}\alpha,\emptyset)=\emptyset$. Further, let $x\in (m,\emptyset)$. Because of $Z_{1}(\alpha_{k}\cdots\alpha_{n})\geq (\im(\alpha_{0}\cdots \alpha_{k-1})\setminus Z_{1})(\alpha_{k}\cdots \alpha_{n})$, we obtain $x(\alpha_{k}\cdots\alpha_{n})^{-1} \subseteq Z_{1}\subseteq X_{1}$. This shows that
$$Y_{1} = (m,\emptyset)(\alpha_{k}\cdots \alpha_{n})^{-1} \subseteq X_{1}.$$
Because of $Y_{1}\subseteq \im(\alpha_{0}\cdots \alpha_{k-1})$ and since $\alpha|_{X_{1}}$ is injective, we can conclude that $(\alpha_{k}\cdots \alpha_{n})|_{Y_{1}}$ is injective. This implies that $\alpha_{k}|_{Y_{1}}$ is injective. On the other hand $Y_{1}$ is order-isomorphic to the open convex set $(m,\emptyset)$.
Hence, $\varphi|_{Y_{1}}$ is an order-isomorphism.

Suppose by way of contradiction that $(Y_{1}\varphi,\emptyset )\neq \emptyset$. Then there is $u \in (Y_{1}\varphi,\emptyset)$. Then $k < n$ (since $(X_{1}\alpha,\emptyset)=\emptyset$) and
$u(\alpha_{k+1}\cdots \alpha_{n})\geq Y_{1}\alpha_{k}(\alpha_{k+1}\cdots \alpha _{n})$ since $u > Y_{1}\varphi$ and $\alpha_{k+1}\cdots \alpha_{n}$ is order-preserving.
Since $(m,\emptyset)\subseteq Y_{1}(\alpha_{k}\cdots\alpha_{n})$ and $(\alpha_{k+1}\cdots \alpha_{n})|_{X_{1}}$ is order-preserving, we can conclude that $u(\alpha_{k+1}\cdots \alpha_{n})\geq (m,\emptyset)$, a contradiction.

Let $\beta \in \OP(X)\setminus \O(X)$ with the ideal $X_{1}$ such that $\beta|_{X_{2}}$ is injective. Then there are $\beta_{1},\ldots,\beta_{l} \in \O(X)\cup \{\varphi\}$ such that $\beta = \beta_0\beta_{1}\cdots \beta_{l}$, where $\beta_0 = \id_X$. Since $\beta \notin \O(X)$, there is $q \in \{1,\ldots,l\}$ with $\beta_{q} = \varphi$. Without loss of generality, we can assume that $q$ is the least $p \in \{1,\ldots,l\}$ with $\beta_{p} = \varphi$ and $\im(\beta_{0}\cdots \beta_{p-1}) \cap X_2 \neq \emptyset$.

We put $Z_{2} = \im(\beta_{0}\cdots \beta_{q-1})\cap X_{2}$ and observe that $Z_{2}(\beta_{0}\cdots \beta_{q-1})^{-1}\geq (\im(\beta_{0}\cdots \beta_{q-1})\setminus Z_{2})(\beta_{0}\cdots \beta_{q-1})^{-1}$. Let $h \in Z_{2}(\beta_{0}\cdots \beta_{q-1})^{-1}$.
Further, let $x \in (h,\emptyset)$. Because of $Z_{2}(\beta_{0}\cdots \beta_{q-1})^{-1}\geq (\im(\beta_{0}\cdots \beta_{q-1})\setminus Z_{2})(\beta_{0}\cdots \beta_{q-1})^{-1}$, we obtain $x(\beta_{0}\cdots \beta_{q-1})\in Z_{2} = \im(\beta_{0}\cdots \beta_{q-1})\cap X_{2}$, i.e., $x(\beta_{0}\cdots \beta_{q-1})\in X_{2}$.
This shows that $(h,\emptyset)\subseteq X_{2}$. Since $\beta|_{X_2}$ is injective, we have that $(\beta_{0}\cdots \beta_{q-1})|_{(h,\emptyset)}$ is injective and
$\beta_{q}|_{(h,\emptyset)(\beta_{0}\cdots \beta_{q-1})}$ is injective. Hence, $Y_{2} = (h,\emptyset)(\beta_{0}\cdots \beta_{q-1})$ is order-isomorphic to the open convex set $(h,\emptyset)\subseteq X_{2}$ and $\varphi|_{Y_{2}}$ is an order-isomorphism.
\end{proof}
~~\\

One can dually consider the case where $X$ has a maximum but no minimum, i.e., the following:
\begin{proposition}\label{pr6}  \rm
Suppose $X$ has a maximum $b$ but no minimum, i.e., $X = (\emptyset,b]$. Let $\varphi \in \OP(X)\setminus \O(X)$ with the ideal $X_{1}$.
Then $\left\langle \O(X), \varphi \right\rangle = \OP(X)$ if and only if, for $i\in \{1,2\}$, there
is a set $Y_{i}\subseteq X_{i}$ which is order-isomorphic to an open convex
subset of $X$ such that $\varphi|_{Y_{i}}$ is an order-isomorphism and $(\emptyset, Y_{2}\varphi) = \emptyset$.
\end{proposition}

\section{The relative rank of $\OP(X,Y)$ modulo $\O(X,Y)$}

In this section, we determine the relative rank of $\OP(X,Y)$ modulo $\O(X,Y)$ for certain proper subsets $Y$ of $X$.
Let $Y$ be a proper convex subset of $X$.
Notice that, the convex subsets of $Y$ are also convex subsets of $X$ and thus, any two open convex subsets of $Y$ are order-isomorphic.

\begin{proposition} \label{pr5} \rm
If $Y$ has no minimum and no maximum then $\rank(\OP(X,Y):\O(X,Y))$ is infinite.
\end{proposition}

\begin{proof}
Suppose by way of contradiction that there is a finite set $G\subseteq \OP(X,Y)\setminus \O(X,Y)$ such that
$\OP(X,Y)=\left\langle \O(X,Y), G\right\rangle$. From Corollary \ref{co1} and since $G$ is a finite set, it follows that
there are $a, b \in Y$ such that either $\{y \in Y : y < a\}\cap \im\sigma =\emptyset$ or $\{y \in Y : y > b\}\cap \im\sigma
=\emptyset$ for all $\sigma \in G$.

Since $Y$ is a proper convex subset of $X$, we have that there is $x\in X$ with $x<Y$ or $Y<x$.
First, we assume that there is $x\in X$ with $x<Y$.

Let $b < h \in Y$ and let $\alpha \in \OP(X,Y)\setminus (\O(X,Y) \cup G)$ with $\im\alpha =\{y \in Y : y \leq h\}$.
Then there are $\alpha_{1},\ldots,\alpha_{k}\in \O(X,Y)\cup G$ such that
$\alpha =\alpha_{1}\cdots \alpha_{k}$, for a suitable $k > 1$ (since $\alpha \notin \O(X,Y) \cup G$). Because $\im\alpha \subseteq \im\alpha_{k}$, we can conclude that
$\alpha_{k}\notin G$, i.e., $\alpha_{k}\in \O(X,Y)$.
Then for $x<Y$, it follows that $x\alpha_{k}\leq Y\alpha_{k}$. Moreover, since $k > 1$ and $\im\alpha_{k-1} \subseteq Y$, we have that $\im(\alpha_{k}|_{Y}) \supseteq \im\alpha = \{y \in Y : y \leq h\}$. Therefore, we obtain $x\alpha_{k}\leq \{y \in Y : y \leq h\}$. Since $Y$ has no minimum, we obtain $x\alpha_{k}\notin Y$, a contradiction.

One can dually consider the case if there is $x\in X$ with $Y<x$. Then for $a > l \in Y$ and $\alpha \in \OP(X,Y)\setminus (\O(X,Y)\cup G)$ with $\im\alpha =\{y \in Y : y \geq l\}$, one obtains also a contradiction.
\end{proof}
~~\\

Further, we will consider two cases with finite relative rank. First, suppose $X$ has no maximum and no minimum.

\begin{theorem}\label{th1} \rm
Suppose $X$ has neither a maximum nor a minimum.
If $Y$ has a minimum or a maximum then $\rank(\OP(X,Y):\O(X,Y)) = 1$.
\end{theorem}

\begin{proof}
Suppose that $Y$ has a maximum but no minimum.

Let $\alpha \in \OP(X,Y)\setminus \O(X,Y)$ with the ideal $X_1$.
Then $X_2\alpha \leq X_1\alpha$. We have that $X_2\alpha$ has a maximum $y_0$ or $X_1\alpha$ has a minimum $y_0$ or
we can extend $X_2\alpha < X_1\alpha$ to a decomposition $Y_{1} < Y_{2}$ with $X_1\alpha \subseteq Y_{2}$ and
$X_2\alpha \subseteq Y_{1}$. Then $Y_{1}$ has a maximum $y_{0}$ or $Y_{2}$ has a minimum $y_{0}$.
In particular, we have $h\in Y$ with $X_2\alpha \leq h \leq X_1\alpha$.

We consider now the set $\widehat{X}$ which is the set $X$ equipped with an additional maximum $x_{\max}$.

Let $\beta :\widehat{X}\rightarrow Y$ with $x\beta =x\alpha$ for all $x\in X$ and $x_{\max}\beta = h$. Clearly, $\beta \in \OP(\widehat{X},Y)$, $\beta|_{X}=\alpha$ and $\alpha = \id_X\beta$.

Let $Z$ be a convex subset of $Y$ such that there is an order-isomorphism $\mu : Z\rightarrow Y$ and there are $y^{-},y^{+}\in Y$ with $y^{-} < Z \leq y^{+}$.
Since $Y$ has a maximum and $\mu$ is a bijection, we have that $Z$ has also a maximum. Let $\max Z = z_{\max}$.

Further, we take an order-isomorphism $\nu :\widehat{X}\rightarrow Z$. In particular, $\nu \in \O(\widehat{X},Y)$. Note that $\id_{X}\nu \in O(X,Y)$.
Let $\delta = \nu^{-1}\beta\mu^{-1}.$  It is easy to verify that $\delta \in \OP(Z)$ and $\beta =\nu\delta\mu$.
From Proposition \ref{pr6}, it follows that there are $\delta_{1}, \ldots, \delta_{k} \in
\O(Z)\cup \{\gamma\}$ such that $\delta = \delta_{1}\cdots\delta_{k}$ for a certain transformation $\gamma \in \OP(Z)\setminus \O(Z)$.

For $i \in \{1, \ldots, k\}$, we can extend $\delta_{i}$ to a transformation
$\tilde{\delta}_{i}\in \OP(X,Y)$.
If $\delta_{i}\in \O(Z)$ then we construct $\tilde{\delta}_{i}$ in the
following way:
$$x\tilde{\delta}_i = \left\{
\begin{array}{ll}
y^{-}, &  x \in (\emptyset,Z) \\
x\delta_i, &  x \in Z \\
y^{+}, &  x \in (Z,\emptyset). \\
\end{array}
\right.$$
Let $x, y \in X$ with $x < y$. If $x, y \in (\emptyset,Z)$ or $x, y \in (Z,\emptyset)$ or $x, y \in Z$
then $x\tilde{\delta}_i = y\tilde{\delta}_i \in \{y^-, y^+\}$ or $x\tilde{\delta}_i \leq y\tilde{\delta}_i$.
The latter holds because of $\delta_{i}\in \O(Z)$. Let now either $x \in (\emptyset,Z)$ and $y \in Z \cup (Z,\emptyset)$ or
$x \in Z$ and $y \in (Z,\emptyset)$.
Since $\delta_{i}\in \O(Z)$, it follows that $Z\delta_{i}\subseteq Z$. Moreover, we have $y^{-} < Z \leq y^{+}$. This implies
$(\emptyset,Z)\tilde{\delta}_i = y^{-} < Z\tilde{\delta}_{i} = Z\delta_{i} \leq y^{+} = (Z,\emptyset)\tilde{\delta}_i$ and thus $x\tilde{\delta}_i \leq y\tilde{\delta}_i$.
Therefore, we obtain $\tilde{\delta}_{i}\in \O(X,Y)$.
It is easy to see that $\tilde{\delta}_{i}|_Z = \delta_{i}$, i.e., $\delta_i = \id_Z\tilde{\delta}_i$.
Notice that $\tilde{\delta}_i\mu \in \O(X,Y)$.\\
If $\delta_{i}=\gamma$ then we define $\tilde{\delta}_i = \tilde{\gamma} : X \rightarrow Y$ by
$$x\tilde{\gamma} = \left\{
\begin{array}{ll}
x\gamma, &  x \in Z \\
(z_{\max})\gamma, &  \text{otherwise.}\\
\end{array}
\right.$$
Clearly, $\tilde{\gamma} \in \OP(X,Y)$, $\tilde{\gamma}|_{Z}=\gamma$, i.e., $\gamma = \id_Z\tilde{\gamma}$.

Finally, we have
$$\alpha = \id_{X}\beta = \id_{X}\nu\delta\mu = \id_{X}\nu\delta_{1}\cdots\delta_{k}\mu =
\id_{X}\nu\id_Z\tilde{\delta}_1\id_{Z}\tilde{\delta}_2\cdots\id_Z\tilde{\delta}_k\mu.$$
Since $\id_{X}\nu\id_Z = \id_{X}\nu$, $\id_{X}\nu\tilde{\delta}_1\id_Z = \id_{X}\nu\tilde{\delta}_1$, $\ldots$,
$\id_{X}\nu\tilde{\delta}_1\cdots\tilde{\delta}_{k-1}\id_Z = \id_{X}\nu\tilde{\delta}_1\cdots\tilde{\delta}_{k-1}$,
we have $\alpha = \id_{X}\nu\tilde{\delta}_1\cdots\tilde{\delta}_k\mu$.
Moreover, we have $\id_{X}\nu \in \O(X,Y)$, $\tilde{\delta}_i \in \O(X,Y)$, if
$i\in \{1,\ldots,k\}$ with $\delta_{i}\in \O(Z)$, and $\tilde{\delta}_{i}=\tilde{\gamma}$
if $i\in \{1,\ldots,k\}$ with $\delta_{i}=\gamma$. Moreover, $\tilde{\delta}_{k}\mu \in \O(X,Y)$
if $\delta_{k}\in \O(Z)$ and $\tilde{\delta}_{k}\mu \in \OP(X,Y)$ otherwise.
Therefore, $\alpha = (\id_{X}\nu)\tilde{\delta}_1\cdots\tilde{\delta}_{k-1}(\tilde{\delta}_k\mu) \in \left\langle \O(X,Y), \tilde{\gamma}, \tilde{\gamma}\mu \right\rangle.$

We define the transformation $\eta : X \rightarrow Y$ by
$$x\eta = \left\{
\begin{array}{ll}
y^{-}, &  x \in (\emptyset,Y) \\
x\mu^{-1}, &  x \in Y \\
y^{+}, &  x \in (Y,\emptyset). \\
\end{array}
\right.$$
Clearly, $\eta \in \O(X,Y)$ and $\eta|_Y = \mu^{-1}$. Therefore, we obtain $\tilde{\gamma} = \tilde{\gamma}\id_Z = \tilde{\gamma}\mu\mu^{-1} = \tilde{\gamma}\mu(\eta|_Y) = \tilde{\gamma}\mu\eta \in \langle \O(X,Y), \tilde{\gamma}\mu \rangle$.

This shows that $\{\tilde{\gamma}\mu\}$ is a relative generating set for $\OP(X,Y)$ modulo $\O(X,Y)$ and thus $\rank(\OP(X,Y):\O(X,Y)) = 1$.

One can dually consider the case where $Y$ has a minimum but no maximum and show that $\rank(\OP(X,Y):\O(X,Y)) = 1$.

Suppose that $Y$ has both a maximum and a minimum. Then we put $Z = Y$, $y^- = y_{\min}$, $y^+ = y_{\max}$.
Clearly $\mu = \id_Y$, $\tilde{\gamma}\mu = \tilde{\gamma}$ and $\alpha = (\id_{X}\nu)\tilde{\delta}_1\cdots\tilde{\delta}_k \in \left\langle \O(X,Y), \tilde{\gamma} \right\rangle$.
Therefore, in this case $\{\tilde{\gamma}\}$ is a relative generating set for $\OP(X,Y)$ modulo $\O(X,Y)$ and $\rank(\OP(X,Y):\O(X,Y)) = 1$.
\end{proof}
~~\\

Now, we consider the case where $X$ has a minimum or a maximum.

\begin{theorem}\label{th2} \rm
Let $X$ has a minimum or a maximum. If $Y$ has a minimum or a maximum then $\rank(\OP(X,Y):\O(X,Y)) = 1$.
\end{theorem}

\begin{proof}
Let $X$ has a minimum $a$ but no maximum.
Let $\widetilde{X}$ be a proper convex subset of $X$ with $a\notin \widetilde{X}$ and $(\widetilde{X},\emptyset)=\emptyset$, which is order-isomorphic to $X$ with
the order-isomorphism $\mu : X \rightarrow \widetilde{X}$.
Further, let $\widehat{X} = X \setminus \{a\}$.
Let $\alpha \in \OP(X,Y)\setminus \O(X,Y)$. Then we define
$\beta : \widehat{X}\rightarrow Y\mu$ by
$$x\beta =\left\{
\begin{array}{cc}
x\mu^{-1}\alpha\mu  & \mbox{ for } x\in \widetilde{X} \\
a\alpha\mu  & \mbox{ otherwise. }
\end{array}
\right.$$
It is easy to see that $\beta \in \OP(\widehat{X}, Y\mu)$. We also see that $\beta|_{\widetilde{X}} = \mu^{-1}\alpha\mu $. This provides
$\alpha = \mu(\beta|_{\widetilde{X}})\mu^{-1} = \mu\beta\mu^{-1}$. Note that $\widehat{X}$ has neither a maximum nor a minimum and that $Y\mu$ has a minimum or a maximum. We can apply Theorem \ref{th1} and obtain that there are
$\beta_{1},\ldots,\beta_{n}\in \O(\widehat{X}, Y\mu)\cup \{\gamma\}$ for a certain $\gamma \in \OP(\widehat{X}, Y\mu)\setminus \O(\widehat{X}, Y\mu)$
such that $\beta =\beta_{1}\cdots \beta_{n}$. Further, we have $\mu^{-1}\mu = \id_{\widetilde{X}}$. Since $\im\beta_{i}\subseteq \widetilde{X}$ for $1\leq i\leq n$, we obtain $$\beta = \beta_{1}\mu^{-1}\mu\beta_{2}\mu^{-1}\mu \cdots \beta_{n-1}\mu^{-1}\mu\beta_{n}$$
and thus
$$\alpha = \mu\beta\mu^{-1} = \mu\beta_{1}\mu^{-1}\mu\beta_{2}\mu^{-1}\mu \cdots \beta_{n-1}\mu^{-1}\mu\beta_{n}\mu^{-1}.$$
For $i \in \{1,\ldots,n\}$, we have $\im\beta_{i}\subseteq Y\mu$ and $\im(\beta_{i}\mu^{-1})\subseteq Y\mu\mu^{-1}= Y\id_X = Y$.
Hence, $\mu \beta_{i}\mu^{-1}\in \OP(X,Y)$. Further, if $\beta_{i}\in \O(\widehat{X}, Y\mu)$ then it is easy to verify that $\mu\beta_{i}\mu^{-1}\in \O(X,Y)$.
Therefore, $\alpha \in \langle \O(X,Y), \mu\gamma\mu^{-1} \rangle$, i.e., $\{\mu\gamma\mu^{-1}\}$ is a relative generating set of $\OP(X,Y)$ modulo $\O(X,Y)$.

One can dually consider the case where $X$ has a maximum but no minimum and show that $\rank(\OP(X,Y):\O(X,Y)) = 1$.

Suppose that $X$ has both a minimum $a$ and a maximum $b$, i.e., $X = [a,b]$. Then we take $\widetilde{X}$ to be a proper convex subset of $X$ with $a, b \notin \widetilde{X}$, which is order-isomorphic to $X$ with the order-isomorphism $\mu : X \rightarrow \widetilde{X}$. Moreover, we put $\widehat{X} = X \setminus \{a, b\}$ and
$\beta \in \OP(\widehat{X}, Y\mu)$ with
$$x\beta =\left\{
\begin{array}{cc}
a\alpha\mu  & x < \widetilde{X} \\
x\mu^{-1}\alpha\mu  & \mbox{ for } x\in \widetilde{X} \\
b\alpha\mu  & x > \widetilde{X}.\\
\end{array}
\right.$$
Using the same method as above, we obtain that $\alpha = \mu\beta\mu^{-1} \in \langle \O(X,Y), \mu\gamma\mu^{-1}\rangle$ and thus, $\rank(\OP(X,Y):\O(X,Y)) = 1$.
\end{proof}
~~\\

Finally, we consider the case where $Y$ has a minimum and a maximum.

\begin{theorem}\label{th3} \rm
Let $X$ has a minimum or a maximum and let $Y$ be a subset of $X$ with a minimum
and a maximum such that $X$ is order-isomorphic to a subset of $Y$. Then $\rank(\OP(X,Y):\O(X,Y)) = 1$.
\end{theorem}

\begin{proof}
Let $\widetilde{Y}\subseteq Y$ such that $X$ is order-isomorphic to $\widetilde{Y}$. By Theorem \ref{th2}, there is $\gamma \in \OP(X, \widetilde{Y})$ such that $\OP(X,\widetilde{Y}) = \langle \O(X,\widetilde{Y}),\gamma \rangle$. Since $\OP(X,\widetilde{Y})\subseteq \OP(X,Y)$, we have $\gamma \in \OP(X,Y)$. We will show that $\OP(X,Y)\subseteq \langle \O(X,Y),\gamma\rangle$. For this let $\alpha \in \OP(X,Y)$ with the ideal $X_{1}$. Then there
are order-isomorphisms
$$\begin{array}{ccc}
\varphi_{1}: & X_{1}\rightarrow \widetilde{Y} & \text{and} \\
\varphi_{2}: & X_{2}\rightarrow \widetilde{Y} &
\end{array}$$
such that $X_{2}\varphi_{2} < X_{1}\varphi_{1}$. We define a transformation $\mu : X\rightarrow \widetilde{Y}$ by
$$x\mu = \left\{
\begin{array}{ll}
x\varphi_{1}, & x \in X_{1} \\
x\varphi_{2}, & x \in X_{2}.
\end{array}%
\right.$$
Clearly, $\mu \in \OP(X,\widetilde{Y})$ with the same ideal $X_1$ as $\alpha$. If $X_{1}$ has a minimum then we put $m =(\min X_{1})\alpha$.
Otherwise, we put $m = (\max X_{2})\alpha$. Thus, $X_{2}\alpha \leq m \leq X_{1}\alpha$. Now, we define a transformation $\widetilde{\mu}: X \rightarrow Y$ by
$$x\widetilde{\mu} = \left\{
\begin{array}{ll}
y_{\min}, & x < X_{2}\varphi_{2} \\
x\varphi_{2}^{-1}\alpha, & x \in X_{2}\varphi_{2} \\
m, & X_{2}\varphi_{2} < x < X_{1}\varphi_{1} \\
x\varphi_{1}^{-1}\alpha, & x \in X_{1}\varphi_{1} \\
y_{\max}, & x > X_{1}\varphi_{1},
\end{array}
\right.$$
where $y_{\min}$ and $y_{\max}$ is the minimum and the maximum, respectively, in $Y$.
Since $y_{\min}\leq Y \supseteq X_{2}\varphi_{2}\varphi_{2}^{-1}\alpha = X_{2}\alpha \leq m \leq X_{1}\alpha = X_{1}\varphi_{1}\varphi_{1}^{-1}\alpha \subseteq Y \leq y_{\max}$, $\varphi_{1}$ and $\varphi_{2}$ are order-isomorphisms, and $\alpha$ is order-preserving both on $X_1$ and on $X_2$, we have $\widetilde{\mu} \in \O(X,Y)$.
From $\mu \in \OP(X,\widetilde{Y}) = \langle \O(X,\widetilde{Y}),\gamma \rangle \subseteq \langle \O(X,Y),\gamma \rangle$, we obtain $\mu\widetilde{\mu} \in
\langle \O(X,Y),\gamma \rangle$. It remains to show that $\alpha = \mu\widetilde{\mu}$. For this let $x \in X_{1}$. Then $x\mu\widetilde{\mu} = x\varphi_{1}\widetilde{\mu} = x\varphi_{1}\varphi_{1}^{-1}\alpha = x\id_{X_1}\alpha = x\alpha$. Analogously, we obtain $x\mu\widetilde{\mu} = x\alpha$, whenever $x \in X_{2}$.

Therefore, $\alpha \in \langle \O(X,Y), \gamma \rangle$, i.e., $\{\gamma\}$ is a relative generating set of $\OP(X,Y)$ modulo $\O(X,Y)$. Since $\O(X,Y)$ is a proper subsemigroup of $\OP(X,Y)$ we have $\rank(\OP(X,Y) : \O(X,Y)) = 1$.
\end{proof}


\begin{thebibliography}{99}

\bibitem{CH} Catarino P. M., Higgins P. M., The monoid of
orientation-preserving mappings on a chain, Semigroup Forum,
58 (1999), 190--206.

\bibitem{DFK} Dimitrova I., Fernandes V. H., Koppitz J., A note on generators of the endomorphism
semigroup of an infinite countable chain, Journal of algebra and its applications, 16(2) (2017), 1750031.

\bibitem{DKT} Dimitrova I., Koppitz J., Tinpun K.,
On the relative rank of the semigroup of orientation-preserving
transformations with restricted range, Proceedings of the
47-th Spring Conference of the Union of Bulgarian Mathematicians, (2018), 109--114.

\bibitem{DK1} Dimitrova I., Koppitz J., On relative ranks of finite transformation
semigroups with restricted range, arXiv:2006.07724 [math.RA], (2020).

\bibitem{DK2} Dimitrova I., Koppitz J., On relative ranks of the semigroup of orientation-preserving
transformations on infinite chains, arXiv:2006.07661v2 [math.RA], (2020).

\bibitem{FHQS1} Fernandes V. H., Honyam P., Quinteiro T. M., Singha B.,
On semigroups of endomorphisms of a chain with restricted range,
Semigroup Forum, 89 (2014), 77--104.

\bibitem{FHQS2} Fernandes V. H., Honyam P., Quinteiro T. M., Singha B.,
On semigroups of orientation-preserving transformations with restricted range,
Comm. Algebra, 44 (2016), 253--264.

\bibitem{FJS} Fernandes V. H., Jesus M. M., Singha B., On orientation-preserving
transformations of a chain, arXiv:1806.08440v2 [math.RA], (2020).

\bibitem{FS} Fernandes V. H., Sanwong J., On the rank of semigroups of transformations
on a finite set with restricted range, Algebra Colloq., 21 (2014), 497--510.

\bibitem{GH} Gomes G. M. S., Howie J. M., On the rank of
certain semigroups of order-preserving transformations, Semigroup
Forum, 51 (1992), 275--282.

\bibitem{GH1} Gomes G. M. S., Howie J. M., On the ranks of certain finite semigroups of
transformations, Math. Proc. Cambridge Philos. Soc., 101 (1987), 395--403.

\bibitem{HMR} Higgins P. M., Mitchell J. D., Ru\v{s}kuc N.,
Generating the full transformation semigroup using order preserving mappings,
Glasgow Math. J., 45 (2003), 557--566.

\bibitem{HMcF} Howie J. M., McFadden R. B., Idempotent rank in finite full transformation
semigroups, Proc. Roy. Soc. Edinburgh, 114A (1990), 161--167.

\bibitem{HRH} Howie J. M., Ru\v{s}kuc N., Higgins P. M.,
On relative ranks of full transformation semigroups, Comm. Algebra, 26 (1998), 733--748.

\bibitem{McA} McAlister D., Semigroups generated by a group and an idempotent,
Comm. in Algebra, 26 (1998), 515--547.

\bibitem{Ruskuc} Ru\v{s}kuc N., On the rank of completely 0-simple semigroups,
Math. Proc. Cambridge Philos. Soc., 116 (1994), 325--338.

\bibitem{Symons} Symons J. S. V., Some results concerning a transformation semigroup,
J. Austral. Math. Soc., 19 (1975), 413--425.

\bibitem{TK1} Tinpun K., Koppitz J., Relative rank of the finite full transformation semigroup
with restricted range, Acta Mathematica Universitatis Comenianae, 85(2) (2016), 347--356.

\bibitem{TK2} Tinpun K., Koppitz J., Generating sets of infinite full
transformation semigroups with restricted range, Acta Scientiarum
Mathematicarum, 82 (2016), 55--63.

\end{thebibliography}
\end{document}